\documentclass{amsart}[12pt]

\usepackage{amsmath,amscd}
\usepackage[]{listings}
\usepackage[]{float}
\usepackage[]{graphicx}
\usepackage[]{url}
\usepackage[]{mathtools}
\usepackage{enumerate}
\usepackage{multicol}

\usepackage[dvipspages,svgnames]{xcolor}
\usepackage{tikz}
\usetikzlibrary{calc}
\usetikzlibrary{backgrounds}
\usetikzlibrary{shapes}
\usetikzlibrary{shapes.multipart}
\usepackage{caption}

\newtheorem{remark}{Remark}

\newtheorem{lemma}{Lemma}

\definecolor{darkgreen}{rgb}{0.0, 0.27, 0.13}
\definecolor{forestgreen}{rgb}{0.0, 0.27, 0.13}
\definecolor{shellcolor}{RGB}{205,192,176}
\definecolor{bashcolor}{RGB}{205,192,176}
\definecolor{makecolor}{RGB}{192,192,192}
\definecolor{flexcolor}{RGB}{180,205,205}
\definecolor{yacccolor}{RGB}{205,181,205}
\definecolor{termcolor}{gray}{0.80}
\definecolor{tracecolor}{gray}{0.90}
\definecolor{gcccolor}{RGB}{224,255,255}
\definecolor{algocolor}{RGB}{144,238,144}
\definecolor{pythoncolor}{RGB}{144,238,144}
\definecolor{asmcolor}{RGB}{230,230,250}

\def\N{\mathbb N}

\def\R{\mathcal R}

\def\rmc(#1,#2){RM(#1,#2)}
\def\form(#1,#2){H(#1,#2)}
\def\val(#1,#2){V(#1,#2)}
\def\boole(#1){ B(#1) }
\def\tildeboole(#1){ \widetilde B(#1) }
\def\icsboole(#1){ B^\dag(#1) }
\def\fcsboole(#1){ B^\ddag(#1) }
\def\fd{{\mathbb F}_2}
\def\fdm{{\mathbb F}^m_2}

\def\GL{{\rm GL}\,}
\def\NL{{\rm NL}}

\def\aglm{{\textsc{agl}}(m,2)}
\def\glm{{\textsc{gl}}(m,2)}
\def\aglp#1#2{\textsc{agl}(#1,#2)}
\def\aglseven{\textsc{agl}(7)}
\def\der(#1,#2){{\rm d}_{#2} #1}
\def\Der(#1,#2){{\rm Der}_{#2} #1}

\def\stab(#1){\textsc{stab}(#1)}

\def\stablevel(#1,#2){\textsc{stab}_{#1}(#2)}

\def\wt{{\rm wt}\,}
\def\cls{{\rm n}}
\def\rmq(#1,#2,#3){\rmc(#1,#2)/\rmc(#3,#2)}

\newcommand{\binomial}[2]{\genfrac{(}{)}{0pt}{}{#1}{#2}}

\def\val(#1){{\rm val}(#1)}

\def\agl#1{{\mathfrak #1}}

\def\anf(#1){{\rm anf}(#1)}
\def\fix(#1,#2,#3,#4){{{\rm fix}^{#1,#2}_{#3}}{(#4)}}
\def\classe(#1,#2,#3){{{\mathcal T}(#1,#2,#3)}}

\def\stab(#1){\textsc{stab}(#1)}
\def\stablevel(#1,#2){\textsc{stab}^{#1}(#2)}
\def\stableveldim(#1,#2,#3){\textsc{stab}^{#1}_{#2}(#3)}
\def\stableveldeg(#1,#2,#3,#4){\textsc{stab}^{#1,#2}_{#3}(#4)}
\def\level(#1){\underset{#1}{\sim}}
\def\simbm(#1){{\sim^{#1}}}
\def\simbstm(#1,#2,#3){{\sim^{#3}_{#1,#2}}}
\def\nsimbstm(#1,#2,#3){{{\not\sim}^{#3}_{#1,#2}}}
\def\restrict(#1,#2,#3){#1\vert_{{#2}_{#3}}}

\def\bound(#1,#2){\underset{#1}{\overset{#2}{\sim}}}
\def\modulo(#1,#2){\mod\rmc(#1,#2)}

\def\pow#1.#2{\tiny$10^{#1.#2}$}

\def\level(#1){\underset{#1}{=}}
\def\val(#1){{\rm val}(#1)}

\def\agl#1{{\mathfrak #1}}

\def\anf(#1){{\rm anf}(#1)}
\def\fix(#1,#2,#3,#4){{{\rm fix}^{#1,#2}_{#3}}{(#4)}}
\def\classe(#1,#2,#3){{{\mathcal T}(#1,#2,#3)}}

\def\stab(#1){\textsc{stab}(#1)}
\def\stablevel(#1,#2){\textsc{stab}^{#1}(#2)}
\def\stableveldim(#1,#2,#3){\textsc{stab}^{#1}_{#2}(#3)}
\def\stableveldeg(#1,#2,#3,#4){\textsc{stab}^{#1,#2}_{#3}(#4)}
\def\level(#1){\underset{#1}{\sim}}
\def\bound(#1,#2){\underset{#1}{\overset{#2}{\sim}}}
\def\modulo(#1,#2){\mod\rmc(#1,#2)}

\def\pow#1.#2{\tiny$10^{#1.#2}$}

\lstdefinestyle{gcc}{%
frame=single,
xleftmargin=2em,
language=c,
numberstyle=\footnotesize,
columns=[l]flexible,
tabsize=4,
showtabs=false,
backgroundcolor=\color{gcccolor},
stringstyle=\color{red}\itshape,
showstringspaces=false,
keywordstyle=\color{blue}\bfseries,
commentstyle=\color{forestgreen}\itshape,
emphstyle=\color{red}\bfseries,
emph={[2]os,sys},
emphstyle={[2]\color{pink}\bfseries},
literate={{:cup:}{$\ \cup\ $}1{:lamb:}{$\lambda$}1{:su:}{$s_u$}1{:rond:}{$\circ$}1{:in:}{$\in$}1{:=}{$\gets$}1{:notin:}{$\not\in$}1{:emptyset:}{$\emptyset$}1{:fp:}{f\texttt{'}}1{:Fp:}{F\texttt{'}}1{:fdm:}{$\fdm$}1{:basis:}{$(b_1,\ldots,b_n)$ }1{:state:}{$\forall x\in \langle b_1,\ldots,b_{i-1}\rangle,\  \widehat F(f^{\prime})\circ A^\ast(x)=\widehat F(f)(x)$}1{:finv:}{$ \widehat J(f\prime)=\widehat J(f)$}1{:A*:}{$A^\ast$}1{:derive:}{$\der(f,v)$}1{:restriction:}{$\restrict(g,E,v)$}1{:v:}{$v$}1}
}

\title{Covering radius of $\rmc(4,8)$} 

\date{April 2023}
\author[1]{Valérie Gillot}
\email{valerie.gillot@univ-tln.fr}
\author[]{Philippe Langevin}
\email{philippe.langevin@univ-tln.fr}
\address{Imath, universit\'e de Toulon}
\thanks{This work is partially supported by the French Agence Nationale de la Recherche through the SWAP project under Contract ANR-21-CE39-0012}

\begin{document}

\begin{abstract} 
	We propose an effective version of the lift by derivation, an invariant that allows us  to provide the classification of $\boole(5,6,8)=\rmq(6, 8, 4)$. The
	main consequence is to establish that the covering 
	radius of the Reed-Muller $\rmc(4,8)$ is equal to 26. .
\end{abstract}

\maketitle


\section{Boolean functions and classification}
Let $\fd$ be the finite field of order $2$. Let $m$ be a positive integer. 
We denote $\boole(m)$ the set of Boolean functions $f \colon \fdm \rightarrow\fd$. The Hamming weight of $f$ is denoted by $\wt(f)$. Every Boolean 
function  has a unique algebraic reduced representation :
$$
f(x_1, x_2, \ldots, x_m ) = f(x) = \sum_{S\subseteq \{1,2,\ldots, m\}} a_S X_S,
\quad a_S\in\fd, \ X_S( x ) = \prod_{s\in S} x_s.
$$
The degree of $f$ is the maximal cardinality of $S$ with  $a_S=1$ in
the algebraic form.  The valuation of $f\not=0$, denoted by $\val(f)$,
is the minimal cardinality of $S$ for
which $a_S=1$. Conventionnally,  $\val(0)$ is $\infty$. 
We denote by $\boole(s, t, m)$ the space of Boolean
functions of valuation greater than or equal to $s$ and of degree less 
than or equal to $t$. Note that  $\boole(s,t,m)=\{0\}$ whenever $s>t$.
The affine general linear group $\aglm$ acts naturally on the right
over Boolean functions. The action of $\agl s\in\aglm$
on a Boolean function $f$ is $f\circ \agl s$, the composition of applications. Reducing modulo the space of functions of degree less than $s$, this group also acts on $\boole(s,t,m)$. The classification of $\boole(s,t,m)$ is a prerequisite for our approach. We denote by $\tildeboole(s,t,m)$ a classification of $\boole(s,t,m)$, that is a set of orbit representatives.
The number of classes
of $\boole( s, t, m)$ is denoted by  $\cls(s, t , m)$.

\section{Covering radius of Reed-Muller codes}
A Reed-Muller code of order $k$ in $m$ variables is a code of length $2^m$, dimension $\sum_{i=0}^k \binomial m i$ and minimal distance $2^{m-k}$.  The codewords correspond to the evaluation over $\fdm$ of Boolean functions of degree less or equal to $k$, we identify the code to the space~: $$\rmc(k,m)=\{ f\in \boole(m) \mid \deg(f) \leq k\}.$$ 

The covering radius $\rho(k,m)$ of $\rmc(k,m)$ is $\rho(k,m)=\max_{f\in\boole(m)} \NL_k (f)$,
where $\NL_k(f)= \min_{g\in\rmc(k,m)} \wt(f+g)$ is the nonlinearity of order $k$ of $f \in \boole(m)$. 
Classical parameters (length, dimension and minimum distance) of Reed-Muller codes are easy to determine and they all share $\aglm$ as group of automorphisms.
The classical results on covering  radii of Reed-Muller codes are given in \cite[p. 800]{HANDBOOK}. Let us recall the simple however essential Lemma~:
\begin{lemma}\label{RADIUS}
\par\noindent
\begin{enumerate}[(i)]
 \item $2\rho(k,m-1) \leq \rho(k,m)$
 \item $\rho(k-1,m-1) \leq  \rho(k,m)$
 \item  $\rho(k,m)\leq \rho(k,m-1) + \rho(k-1,m-1)$
\end{enumerate}

\end{lemma}

However, most of covering radii are still unknown. 
Recent results are obtained in \cite{GAO,WANG} in the case $m=7$. Therefore, all the covering radii are known for $m\leq 7$. 
For $m=8$,  most the covering radii are unknown. Table \ref{M=8} is an update of Table \cite[p. 802]{HANDBOOK} with the latest results corresponding to cases $m=7, 8$.

\begin{table}[h]
	\caption{\label{M=8} Updated Table of Handbook of coding theory.
}
\begin{tabular}{ | c || c | c | c | c | c | c | c | c |}
\hline
 $k$ & 1 & 2 & 3 & 4 &5 & 6 & 7 & 8 \\
 \hline
 $\rho(k,8)$ & 120 & $88^a$ -- 96 & $50^b$ -- $67^f$ & \textbf{26}$^c$ & 10  & 2 & 1 &0 \\
 \hline 
 $\rho(k,7)$ & 56 & $40^d$ & $20^e$ & 8  & 2 & 1 &0& \\
 \hline
\end{tabular}
\end{table}
\begin{itemize}
  \item[(a)] One can check the non-linearity of order $2$ of $abd+bcf+bef+def+acg+deg+cdh+aeh+afh+bfh+efh+bgh+dgh$ is 88 ;
  \item[(b)] The lower bound is a consequency of the classification of $\boole(4,4,8)$, see \cite{GLPL}; 
  \item[(c)]  Obtained in this paper as a consequence of a lower bound found in \cite{RANDALL};
    \item[(d)] See the result in \cite[Theorem 11]{WANG} ;
     \item[(e)] See the result in \cite[Theorem 1]{GAO} ;
     \item[(f)] Consequence of Lemma \ref{RADIUS}-(iii).
 \end{itemize}

We also consider  $\rho_t(k,m)$ the relative covering radius of $\rmc(k,m)$ into $\rmc(t,m)$, 
\begin{equation}
\label{RHO}\rho_t(k,m)=\max_{f\in\rmc(t,m)} \NL_k (f)= \max_{f\in \boole(k+1,t,m)} \NL_{k} (f)
\end{equation}

In the paper \cite{RANDALL}, the authors present methods for computing the distance from a Boolean function in $\boole(m)$ of degree $m-3$ to the Reed-Muller space $\rmc(m-4,m)$. It is useful to determine the relative covering radius $\rho_{m-3}(m-4,m)$. In particular, their result $\rho_5(4,8)=26$ is a milestone for our purpose : computation of $\rho(4,8)$. It is necessary to determine $\rho_6(4,8)$, but considering the formula (\ref{RHO}) the cardinality of $\boole(5,6,8)=2^{84}$ is too large, using a set of representatives of $\boole(5,6,8)$ 
$$\rho_6(4,8) = \max_{f\in \tildeboole(5,6,8)} \NL_{4} (f).$$
Hence, the search space is reduced to the $20748$ Boolean functions.

Our strategy for determining the covering radius $\rho(4,8)$ is described in figure \ref{FIG}. It consists in two parts. A first part dedicated to the tools which allow to obtain the classification of $\boole(5,6,8)$ : cover set, invariant and equivalence. A second part is dedicated to the estimation of the 4th order nonlinearity of element in $\tildeboole(5,6,8)$.

\section{Cover set and classification}
Given a set of orbit representatives $\tildeboole(s,t,m) $ of $\boole(s,t,m)$ under the action of $\aglm$, we determine $\rho_t(s-1,m)$ :
 $$\rho_t(s-1,m) = \max_{f \in\boole(s,t,m)} \NL_{s-1} (f)= \max_{f\in \tildeboole(s,t,m)} \NL_{s-1} (f).$$
In general, the determination of a $\tildeboole(s,t,m)$ is hard computational task. So, we introduce an intermediate concept, a cover set of $\boole(s,t,m)$ is a set containing $\tildeboole(s,t,m)$ and eventually other functions of $\boole(s,t,m)$. In order to obtain a classification from a cover set, we will need a process to eliminate functions in same orbit. In the first instance, we construct a cover set with reasonable size in two reduction steps applied to $\boole(s,t,m)$.
Any Boolean function $f\in\boole(m)$ can be written as $ x_m g + h$ with $g,h\in \boole(m-1)$. 
In particular, 
\begin{equation} \label{DECOMP}
 \boole(s,t,m) = \big\lbrace x_m g + h \mid g\in \boole(s-1,t-1,m-1),\ h\in \boole(s,t,m-1)\big\rbrace.
\end{equation}
\begin{lemma}[Initial cover set]\label{FIRST}
The set
 \begin{equation} \label{CS}
 \icsboole(s,t,m)=\{  x_m g + h  \mid  g\in \tildeboole(s-1,t-1,m-1), h\in \boole(s,t,m-1)\}
\end{equation}
is a cover set of $\boole(s,t,m)$ of size $\sharp \tildeboole(s-1,t-1,m-1) \times \sharp \boole(s,t,m-1)$.
\end{lemma}
\begin{proof}
 An element $\agl s\in \aglp{m-1}2$ acts on $f$ by $x_m g\circ \agl s + h\circ \agl s$.
\end{proof}

\begin{lemma}[Action of stabilizer]\label{COVER}

 Let us fix $g \in \tildeboole(s-1,t-1,m-1)$. 
 \begin{enumerate}
  \item For all $\agl s\in \aglp{m-1}{2}$  in the stabilizer of $g$, the functions $x_m g + h$ and $x_m g + h\circ \agl s$ are in the same orbit.
  \item For all $\alpha \in \rmc(1,m-1)$, the functions $x_m g + h$ and $x_m g + h + \alpha g$ are in the same orbit.
 \end{enumerate}
where orbits correspond to the action of $\aglm$ on $\boole(s,t,m)$.
 
\end{lemma}
\begin{lemma}[Second cover set]\label{SECOND} The set
\begin{equation}\label{CS2}
 \fcsboole(s,t,m)=\bigsqcup_{g\in \tildeboole(s-1,t-1,m-1)}\big\lbrace\ x_m g + h \mid h \in \R(g)\ \big\rbrace. 
\end{equation}
is a cover set of size $\sharp\fcsboole(s,t,m)= \sum_{g\in\tildeboole(s-1,t-1,m-1)}  \sharp\R(g)$.
Denoting by $\R(g)$ an orbit representatives set for  the action over $\boole(s,t,m-1)$ of the group spaned  by the transformations $h \mapsto h\circ \agl s$ and $h \mapsto h+ \alpha g $. 
\end{lemma}
\begin{proof}
For each $g \in \tildeboole(s-1,t-1,m-1)$ apply Lemma \ref{COVER} to the cover set (\ref{CS}).
\end{proof}

In order to determine $\rho_6(4,8)$,   the initial cover is $\icsboole(5,6,8)=\tildeboole(4,5,7) \times \boole(5,6,7)$. The classification $\tildeboole(4,5,7)$ is obtained in \cite{GLPL},  its cardinality is 179, whence $\sharp\icsboole(5,6,8)$  is $179 \times 2^{28} \approx 2^{35.5}$.  

Applying Lemma \ref{SECOND}, we obtain a cover set of size $3828171\approx 2^{21.9}$. It is already known that $\sharp\tildeboole(5,6,8)=20748$, the determination of an orbit representatives set is the subject of the next sections. Our approach is based on invariant tools and equivalence algorithm.

\section{Invariant}

From the result of the previous section in the case $\boole(5,6,8)$, we have to extract $20748$ orbit representatives among    $3828171$ functions. 
Two elements $f, f'\in\boole(s,t,m)$ in the same orbit under the action of $\aglm$ are said equivalent, we denote $f \simbstm(s,t,m) f'$, that means that there exists $\agl s \in \aglm$ such that $f' \equiv f\circ \agl s \mod \rmc(s-1,m)$.
An invariant $j : \boole(s,t,m) \rightarrow X$, for an arbitrary set $X$, satisfies  $f\simbstm(s,t,m) f' \Longrightarrow j(f)=j(f')$. If $j(f)=j(f')$ and $f\nsimbstm(s,t,m) f'$, we say there is a collision.

Let us recall the derivative $\der(f,v)$ of a Boolean function $f$ in the direction $v$ is the application defined by $\fdm \ni x \mapsto \der(f,v)(x)= f(x+v)+ f(x)$.   In the specific case $f\in \boole(s,t,m)$, we define the derivative as $$\Der(f,v)\equiv\der(f,v)\mod \rmc(s-2,m).$$  This derivative is an element of $\boole(s-1,t-1,m)$ and we consider the following map~:
\begin{align*}
F \colon \boole(s,t,m) & \longrightarrow  \tildeboole(s-1,t-1,m)^{\fdm}\\
f & \xmapsto{\phantom{\longrightarrow}} \widetilde{\Der(f,.)},\\
\end{align*}

\begin{lemma}\label{FoA} 
Let be $f\in \boole(m)$, $\agl s \in \aglm$.
Considering the linear part $A\in \glm$ and $a\in \fdm$ the affine part of $\agl s=(A,a)$, $\agl s (x)=A(x)+a$, we have $F(f\circ \agl s)=F(f)\circ A$.
\end{lemma}

\begin{proof} Note that $\agl s(x+y)=A(x+y)+a=\agl s(x)+A(y)$. For $x,v\in \fdm$, $f\in \boole(m)$
 \begin{align*}
  \der((f\circ \agl s),v)(x)&= f\circ \agl s(x+v) + f\circ \agl s (x)\\
  &=f(\agl s(x)+A(v)) + f\circ \agl s (x)\\
  &=(\der(f,{A(v)}))\circ \agl s(x)
 \end{align*}
Reducing modulo $\rmc(s-2,m)$, we have $\Der((f\circ \agl s),v)\equiv (\Der(f,{A(v)}))\circ \agl s$, therefore $\widetilde{\Der((f\circ \agl s),v)}=\widetilde{\Der(f,{A(v)})}$, whence $F(f\circ \agl s)= F(f)\circ A$.
\end{proof}
\begin{lemma}[Invariant]\label{INVARIANT} 
The application $J$ mapping $f\in \boole(s,t,m)$ to the distribution of the values of $F(f)(v)$, for all $v\in \fdm$, is an invariant. 
\end{lemma}
\begin{proof} Let consider $f,f'\in \boole(s,t,m)$, $\agl s \in \aglm$,  such that $f' \equiv f\circ \agl s \mod \rmc(s-1,m)$ (i.e. $f\simbstm(s,t,m) f'$).
Applying Lemma \ref{FoA}, we obtain $F(f')=F(f)\circ A$.
\end{proof}

Let us observe the derivative of $f\in\rmc(t,m)$ in the direction $e_m$, using the  decomposition of $f$ as in (\ref{DECOMP}),  for $(y,y_m)\in \fd^{m-1}\times\fd $ and $e_m=(0,1)\in \fd^{m-1}\times\fd $,  we obtain :
\begin{align*}
\der(f,e_m) (y,y_m)&= f((y,y_m)+(0,1))+f(y,y_m)\\
&= x_m(y,y_m+1)g(y)+ x_m(y,y_m)g(y)+ h(y) + h(y)\\
&= (y_{m}+1) g(y) + y_m g(y)\\
&= g(y)\\
\end{align*}
It is nothing but the partial derivative with respect to $x_m$. Hence,  $g$ is a Boolean function in $m-1$ variables of degree less or equal to $t-1$. This fact holds in general for a derivation in any direction $v$. 
A Boolean function $f\in \boole(m)$ is $v$-periodic iff $ f(x+v)=f(x),\forall x\in \fdm$. The $v$-perodic Boolean functions are invariant under the action of any transvection  $T\in \glm$ of type $T(x)=x+\theta(x)v$, where $v$ is in the kernel of the linear form $\theta$.

For any supplementary $E_v$ of $v$, the restriction $\restrict(f,E,v)$ of a $v$-periodic function $f\in\boole(m)$ is a function in $m-1$ variables.
Note that for $f\in \boole(s,t,m)$, $\Der(f,v)$ is $v$-periodic whoose its restriction  to $E_v$ is a Boolean function in $m-1$ variables of degree less or equal to $t-1$.
\begin{lemma}\label{RESTRICTION} Let be $f, g\in\boole(m)$ two $v$-perodic Boolean functions. If $f$ is equivalent to $g$ in $\boole(m)$ then $\restrict(f,E,v) $ is equivalent to $\restrict(g,E,v)$  in $\boole(m-1)$, for any supplementary $E_v$ of $v$.
\end{lemma}
\begin{proof}
If $f$ and $g$ are equivalent in $\boole(m)$, there exists $\agl s=(A,a)$ such that $f\circ\agl s =g$. The
case of a translation is immediate. We may assume $a=0$ that is the action of the linear part $A$, $f\circ A=g$.
Since $g$ is $v$-perodic, $g$ is fixed by any transvection $T=x+\theta(x)v$ where $v$ is in the kernel of the linear form $\theta$~:
$$\forall x \in \fdm, \quad g(T(x))=g(x+\theta(x)v)=g(x)$$
We denote $P$ the projection  of $\fdm$ over $E_v$ in the direction of $v$
	($P(e+v)=e$),
	$$\forall x\in\fdm,\quad g(x)=g(T(x))=f(AT(x))=f(PAT(x)).$$

Note that  $AT(x)=A(x)+\theta(x)A(v)$. We are going to determine $\theta(A^{-1}(v))$ 
so that $\ker PAT \cap E_v=\{0\}$. That means for  $x\in E_v\setminus \{0\}$, $AT(x)\not\in \{0,v\}$.
Let $x\in\fdm$ such that $AT(x)=\lambda v$ with $\lambda \in \fd$.
 \begin{align*}
  & A(x)+\theta(x)A(v)=\lambda v\\
  & x + \theta(x)v=\lambda A^{-1}( v)\\
   &\theta(x)+\theta(x)\theta(v)=\lambda\theta( A^{-1}(v))\quad \theta(x)= \lambda\theta( A^{-1}(v))\\
  &x= \lambda (  A^{-1}( v) + \theta( A^{-1}(v)) v)\\
  \end{align*}

\noindent There are two cases to be considered~:
\begin{itemize}
 \item $v\in A(E_v)$ : $A^{-1}(v)\not =v$,  we can fix $\theta(A^{-1}(v))=1$. Thus, $x=\lambda (A^{-1}(v) + v)$.
$$
   x=\lambda (A^{-1}(v) + v)\quad\lambda=\begin{cases}
    0 \text{, $x=0$}\\
    1 \text{, $x\not\in E_v$}
      \end{cases}$$
 \item $v\not\in A(E_v)$ :  $A^{-1}(v)\not\in E_v$, we can fix $\theta(A^{-1}(v))=0$. Thus $x =\lambda A^{-1}(v)$, we obtain  $x=0$ for $\lambda =0$  and $x\not\in E_v$ for $\lambda =1$
$$
  x =\lambda A^{-1}( v)\quad
   \lambda=\begin{cases}
    0 \text{, $x=0$}\\
    1 \text{, $x\not\in E_v$}
      \end{cases}$$

\end{itemize}
 In these two cases, we obtain  $x=0$ for $\lambda =0$  and $x\not\in E_v$ for $\lambda =1$.
Hence, the restriction of $PAT$ to $E_v$ is an automorphism, thus, $\restrict(f,E,v) $ 
is equivalent to $\restrict(g,E,v)$  in $\boole(m-1)$. 
\end{proof}

By numbering the elements of $\tildeboole(s-1,t-1,m)$, $F(f)$ takes its values in $\N$. We can consider its Fourier transform  $\widehat F (f)(b)=\sum_{v\in \fdm} F(f)(v) (-1)^{b.v}$. For $A\in \GL(m)$, the relation $F(f')=F(f)\circ A$ becomes $\widehat F(f')\circ A^\ast=\widehat F(f)$, $A^\ast$ is the adjoint of $A$. We denote by $J$ the invariant corresponding to the values distribution  of $F(f)$ and $\widehat J$ the invariant corresponding the values distribution of $\widehat F(f)$.
These invariants $J$ and $\widehat J$ were introduced in \cite{BRIER}. In our context the invariant $\widehat J$ is more discriminating than $J$. The application of Lemma \ref{RESTRICTION} allows us to consider the derivatives functions in $\boole(s-1,t-1,m-1)$ instead of $\boole(s-1,t-1,m)$.

\begin{remark} To make the algorithm \texttt{Invariant}, we need to optimise the class determination of an element of $\boole(4,5,7)$. There is only $4$ classes in $\tildeboole(5,5,7)$.
We precompute the complete classification of $\boole(5,5,7)$ by determining a representatives set $\{r_1,r_2,r_3,r_4\}$ of  $\tildeboole(5,5,7)$, stabilizers  $\{S_1,S_2,S_3,S_4\}$ of each representative and a transversale.  For each stabilizer, we keep in memory a description of the orbits of $\boole(4,4,7)$ under the stabilizer $S_i$. The class of an element $h \in \boole(4,5,7)$ is obtained from a representative $r_i\simbstm(5,5,7) h$ and a transversale element $\agl s \in \aglseven$  such that $h \circ \agl s \equiv r_i \mod \rmc(4,7)$ using a lookup table for the key $h\circ \agl s + r_i$.\\
There is $179$ classes dans $\tildeboole(4,5,7)$. The amount of memory to store this data is about 32 GB.
\end{remark}

\begin{lstlisting}[float=th, style=gcc, xleftmargin={20mm},caption={Invariant},linewidth={110mm}, emph={pop,push,empty}]
Algorithm Invariant( f, s, t, m )
{   // f element of B(s,t,m) 
	for each :v: in :fdm:
		g:= :derive:
		h := :restriction:
		F[:v:] := Class( h, s-1, t-1, m-1 )
	return FourierTransform( F )
}
\end{lstlisting}

Applying the invariant  $J$  to the 3828171 Boolean functions of the cover set $\fcsboole(5,6,8)$,
one finds $20694$ distributions that means there are 54 collisions. On the same set, 
the invariant $\widehat J$ takes  $20742$ values : there are only 6 collisions. In the
next section, we describe an equivalence algorithm to detect and solve theses collisions.

\section{Equivalence}

In this section, we work exclusively in the space  $\boole(t-1,t,m)$, i.e. in the particular case $s=t-1$.
Considering $\widehat J$, the invariant corresponding to the values distribution of $\widehat F(f)$. Two functions $f, f' \in \boole(t-1,t,m)$ that do not have the same values distribution are  not equivalent. In the case $f \simbstm(t-1,t,m) f'$, the distributions are identical and there exists $A \in \glm 
$ such that \begin{equation}\label{ADJOINT} F(f')=F(f)\circ A \quad\text{and}\quad \widehat F(f')\circ A^\ast=\widehat F(f).
\end{equation}
The existence of $A$ does not guarantee the equivalence of the functions. Such an $A$ is said a candidate which must be completed by an affine part $a\in \fdm$ to be able to conclude equivalence. For $f\in\rmc(t,m)$ and $x\in \fdm$,
\begin{align*}
 \der(f,u,v)(x)&=\der((\der(f,v)),u)(x)\\
 &= \der((f(x+v)+f(x)),u)\\
 &= f(x+u+v)+f(x+u)+f(x+v)+f(x)\\
 &=f(x+u+v)+f(x)+f(x+u)+f(x)+ f(x+v)+f(x)\\
 &= \der(f,u+v)(x) + \der(f,u)(x)+\der(f,v)(x)\\
\end{align*}
The degree of $\der(f,u,v)$ is less or equal $t-2$, reducing modulo $\rmc(t-2,m)$, we obtain
$$\der(f,u+v)(x) + \der(f,u)(x)+\der(f,v)(x)\equiv 0.$$
The set $\Delta(f)=\lbrace \der(f,v) \mod \rmc(t-2,m) \mid v \in \fdm\rbrace$ is a subspace of $\boole(t-1,t-1,m)$.

\begin{lemma}[Candidate checking]\label{AFFINEQUIV}
 Let $f, f'$ be in $\boole(t-1,t,m)$.
 Let us consider a candidate $A\in \GL(m)$.
 There exists $a\in\fdm$ such that $f'\equiv f\circ (A,a)\mod \rmc(t-2,m)$ if and only if $f'\circ A^{-1}+f \in \Delta(f)$.
\end{lemma}

\begin{proof}
 If $f'\equiv f\circ (A,a)\mod \rmc(t-2,m)$, there exists $r\in  \rmc(t-2,m)$ such that for all $x\in\fdm$
 \begin{align*}
  f'(x)&= f\circ (A,a) (x)+r(x)= f(A(x)+a)+r(x)\\
  f'\circ A^{-1}(x)&=f(x+a) +r(x)\\
  f'\circ A^{-1}(x)+f(x)&=f(x+a) +f(x)+r(x)\\
  (f'\circ A^{-1}+f)(x)&=\der(f,a)(x)+r(x)\\
  \end{align*}
Thus $f'\circ A^{-1}+f \in \Delta(f)$.
Conversely, for $f'\circ A^{-1}+f \in \Delta(f)$, there exists $a\in \fdm$ such that $f'\circ A^{-1}+f\equiv \der(f,a) \mod \rmc(t-2,m)$. There exists $r\in  \rmc(t-2,m)$ such that for all $x\in\fdm$, $(f'\circ A^{-1}+f)(x)=\der(f,a)(x)+r(x)$. By repeating the calculations in reverse order, we have $f'\equiv f\circ (A,a)\mod \rmc(t-2,m)$.
\end{proof}

From Lemma \ref{AFFINEQUIV}, one deduces an algorithm \texttt{CandidateChecking(A,f,f')} returning  \texttt{true} if  there exists an element $a\in\fdm$ such that  $f'\equiv f\circ (A,a)\mod \rmc(t-2,m)$, \texttt{false} otherwise.
Given $f,f' \in \boole(t-1,t,m)$ satisfying $\widehat J(f)=\widehat J(f')$, the algorithm \texttt{Equivalent(f,f',iter)}\footnote{the parameter \texttt{iter} ranges from $1024$ to $2^{23}$ depending on the situation} tests in two phases if $f$ and $f'$ are equivalent under the action of $\aglm$ modulo $\rmc(t-2,m)$ :
\begin{enumerate}
	\item determine at most \texttt{iter} candidates $A^\ast\in \GL(m)$ such that $\widehat F(f')\circ A^\ast=\widehat F(f)$
 \item For each candidate $A^\ast$, call \texttt{CandidateChecking(A,f,f')}.
\end{enumerate}
The algorithm ends with one of following three values : 
$$ \texttt{Equivalent($f$,$f'$,iter)}=\begin{cases} 
\texttt{NotEquiv},& \text{all potential  $A$ were tested, so $f\nsimbstm(t-1,t,m) f'$} ; \\
  \texttt{Equiv}, & \text{there exists a $(A,a)$ to prove $f\simbstm(t-1,t,m) f'$} ; \\
  \texttt{Undefined}, & \text{\texttt{iter} is too small to conclude}. \\
\end{cases}$$
 
\begin{center}
	\begin{lstlisting}[style=gcc, xleftmargin={20mm},caption={Equivalence in $\boole(t-1,t,m)$ under the action of $\aglm$ },linewidth={110mm}, emph={pop,push,empty}]
Algorithm  Equivalent(f,:fp:, iter)
{	// f,f' given elements of B(t-1,t,m) 
    // satisfying :finv:
    // return Equiv or NotEquiv or Undefined
	s  :=  random element of AGL(m)
	f  :=  f :rond: s
	basis := :basis:   a basis of :fdm:
	flag := NotEquiv
	// determine :A*: in GL(m)
    :A*:(0) := 0
    Search(1,basis)
    return flag
}
\end{lstlisting}
\end{center}

\begin{center}
	\begin{lstlisting}[style=gcc, xleftmargin={20mm},caption={Search},linewidth={110mm}, emph={pop,push,empty}]
Algorithm  Search(i,basis)
{   // basis=:basis: a basis of :fdm: 
    // i index of basis elements in {1,2,...,m}
    if ( i > m ) 
        // :A*: in GL(m) is fully constructed
        // check the existence of a in :fdm:
        if CandidateChecking(A,f,:fp:) 
            flag := Equiv
            return
        iter := iter - 1
        if ( iter < 0 ) 
            flag := Undefined
            return
    else
        // :state:
        // continue construction of :A*:
        for each y in :fdm:
            if Admissible(y,i) and ( flag= NotEquiv )
                Search(i+1,basis)
\end{lstlisting}

\end{center}

The algorithm \texttt{Admissible(y,i)} checks the possible continuation of the construction of $A^\ast$ over $\langle b_1,\ldots,b_{i-1}, b_i\rangle $, setting $A^\ast(x+b_i):=A^\ast(x)+y$ for all $x\in \langle b_1,\ldots,b_{i-1}\rangle$. Then, the function returns \texttt{true} if $\forall x\in \langle b_1,\ldots,b_{i-1}, b_i\rangle,\  \widehat F(f^{\prime})\circ A^\ast(x)=\widehat F(f)(x)$, and \texttt{false} otherwise.

\section{Determination of $\rho(4,8)$}

\noindent The different steps of our strategy to determine $\rho(4,8)$ are sumarised in \ref{FIG}.

\begin{lstlisting}[style=gcc, xleftmargin={20mm},caption={},linewidth={110mm}, emph={pop,push,empty}]
Algorithm NonLinearity(k,m,f,iter,limit)
{
	G := generator matrix of RM(k,m)
	while ( iter > 0 )
		for( i = 0 ; i < k; i++ )
		do {
			p = random( n ) 
		} while ( not G[i][p] )
		for( j = i+1 ; j < k; j++ )
			if ( G[j][p] ) 
			G[j] := G[j] xor G[i]
		if ( f[ p ] ) 
			f  := f  xor G[i]
		w = weight( f )
		if ( w <= limit ) 
			return true
	iter := iter - 1
	return false
}
\end{lstlisting}

This algorithm proceeds ramdom Gaussian eliminations to generate small weight codewords in a translate of $\rmc(k,m)$. 
To dertermine the covering radii $\rho_6(4,8)$ and $\rho (6,8)$, we have to estimate the nonlinearity of order $4$ of some functions in $\boole(8)$. We use the probabilistic algorithm \texttt{NonLinearity} three times :

\begin{enumerate}
 \item to check the non-existence of function in $\tildeboole(5,6,8)$ of nonlinearity of order 4 greater or equal to $28$ ;
 \item to extract the set of two functions $\{f,g\}$ in $\tildeboole(5,6,8)$ with nonlinearity of order $4$ greater or equal to $26$ ;
 \item to prove the nonlinearity of order 4 of the functions $\{f+\delta_a, g+\delta_a\}$ is not greater or equal to $27$.
\end{enumerate}



\subsection{Compute $\rho_6(4,8)$}
Recall that $$\rho_6(4,8) = \max_{f\in \tildeboole(5,6,8)} \NL_{4} (f)= \max_{f\in \tildeboole(5,6,8)}\min_{g\in\rmc(4,8)} \wt(f+g).$$
We apply the algorithm \texttt{NonLinearity} to $\tildeboole(5,6,8)$ to confirm that all these functions have a nonlinearity of order $4$ less or equal to $26$. Using the result $\rho_5(4,8)=26$ of \cite{RANDALL}, we obtain $\rho_6(4,8)=26$.
 
\subsection{Compute $\rho(4,8)$}

 Knowing that $\rho(6,8)=2$ and from the previous result of $\rho_6(4,8) =26$, we have  $$\rho(4,8)\leq \rho_6(4,8) + \rho(6,8) =28.$$
 
A second application of the algorithm \texttt{NonLinearity} eliminates from $\tildeboole(5,6,8)$ $20746$ functions of nonlinearity of order 4 less than 26.  After this process, there are two remaining functions :
\begin{align*}
 f =
abcef+acdef+abcdg&+abdeg+abcfg+acdeh+abcfh\\
&+bdefh+bcdgh+abegh+adfgh+cefgh
\end{align*}
and 
$$g=abcdeh+abcdf+abcef+abdeg+bcefh+adefh+bcdgh+acegh+abfgh
$$
We retrieve the cocubic function $f$, mentioned in \cite{RANDALL}, its degree is $5$ and its nonlinearity of order 4 is $26$. The other function $g$ has degree $6$ and its nonlinearity is probabily $26$ and certainly less or equal to $26$.
Now, we are going to prove that there is no Boolean function in $\boole(8)$ with a nonlinearity of order 4 equal to $28$. For this purpose, it is sufficient to check the non-existence of a function $h$ satisfying $\NL_4(h)=27$, such a $h$ has an odd weight and  therefore its degree is $8$.
For $ a\in \fdm$, we denote by  $\delta_a$ the Dirac function, $\delta_a(x)=1$ iff $x=a$. Every Boolean function can be expressed by a sum of Dirac $f(x)=\sum_{\{a| f(a)=1\}} \delta_a(x)$. The polynomial form  of $\delta_a$ is :
\begin{equation}
\label{DELTA}
\delta_a(X_1,X_2, \ldots, X_m)=(X_1+\bar a_1)(X_2+\bar a_2)\cdots (X_m+\bar a_m)
\end{equation}
                                                                    where $\bar a_i=a_i+1$. 
\begin{lemma}
 An odd weight function is at distance one from $\rmc(m-2,m)$.
\end{lemma}
\begin{proof}
We denote $\widetilde{X_i}$ the monomial term of degree $m-1$ with all variables except $X_i$. Let us consider an odd weight function $h\in \boole(m)$, its degree is $m$, so
$$h(X_1,X_2, \ldots, X_m)= X_1X_2\ldots X_m+ \bar a_1\widetilde{X_1}+\cdots+ \bar a_m\widetilde{X_m} + r(x)$$
 where $\deg(r) \leq m-2$. From (\ref{DELTA}), we also have
$$\delta_a(X_1,X_2, \ldots, X_m)= X_1X_2\ldots X_m+ \bar a_1\widetilde{X_1}+\cdots+ \bar a_m\widetilde{X_m} + r'(x)$$ with $\deg(r')\leq m-2$. We obtain $h\equiv \delta_a \mod \rmc(m-2,m)$. The Dirac function has weight 1,  so the distance of $h$ to $\rmc(m-2,m)$ is 1. 
\end{proof}

A third application of the algorithm \texttt{NonLinearity} to the set $\{f,g\}$ translated by the $256$ Dirac functions give the non-existence of odd weight functions of nonlinearity of order 4 greater or equal to $27$. That means there is no function in $\boole(8)$, with nonlinearity of order 4 greater or equal to $27$  and we obtain $\rho(6,8)=26$.
The second and third applications of the algorithm \texttt{NonLinearity} need $569713$ iterations.

\begin{remark}
The extraction of 20748 classes of $\tildeboole (5,6,8)$ with invariant approach and \texttt{equivalent} algorithm needs several weeks of computation (equivalence test)
\end{remark}

\begin{remark} The number of iterations to estimate the 4th order nonlinearity of Boolean functions 565252 in average. The total running time to check the nonlinearity is about one day using 48  processors.
 
\end{remark}

\begin{figure}[h]
\centering
\begin{tikzpicture}[scale=0.70,transform shape]
 
\node[draw,rectangle split, rectangle split parts=2, rectangle split part fill={red!50,red!20}, text width=4cm,align=center,rounded corners] (dep) at (0,0)
{ $\boole(5,6,8)$
\nodepart{two}
$\boole(4,5,7) \times \boole(5,6,7)$
dimension 84 \\
cardinality = $2^{84}$};

%

\node[draw,rectangle split, rectangle split parts=2, rectangle split part fill={red!50,red!20}, text width=4cm,align=center,rounded corners] (ics) at (+7,0)
{ $\icsboole(5,6,8)$
\nodepart{two}
initial cover set\\
$\tildeboole(4,5,7) \times \boole(5,6,7)$\\
 card $= 179 \times 2^{28} \approx 2^{35.5}$};
\draw[-stealth,Red!80,line width=1pt] (dep.east)--(ics.west) node[sloped,midway,above]{first reduction}node[sloped,midway,below]{Lemma \ref{FIRST} };

%

\node[draw,rectangle split, rectangle split parts=2, rectangle split part fill={red!50,red!20}, text width=4cm,align=center,rounded corners] (fcs) at (+14,0)
{ $\fcsboole(5,6,8)$
\nodepart{two}
second cover set\\
card = $3828171$\\\phantom{card}$\quad\approx 2^{21.9}$};
 \draw[-stealth,Red!80,line width=1pt] (ics.east)--(fcs.west) node[sloped,midway,above]{stabilizers action}node[sloped,midway,below]{Lemma \ref{COVER}\&\ref{SECOND} };

%

\node[draw,rectangle split, rectangle split parts=2, rectangle split part fill={red!50,red!20}, text width=4cm,align=center,rounded corners] (inveq) at (+14,-4)
{ $\tildeboole(5,6,8)$
\nodepart{two}
representatives set\\
card = $20748$};
 \draw[-stealth,red!50,line width=1pt] (fcs.south)--(inveq.north) node[sloped,midway,above]{invariant }node[sloped,midway,below]{equivalence};

\node[draw,rectangle split, rectangle split parts=2, rectangle split part fill={Navy!30,Blue!20}, text width=4cm,align=center,rounded corners] (nl26) at (+7,-4)
{$\{f,g\}$
\nodepart{two}
only 2 functions
have  probably $\NL_4\geq 26$\\ 20746 have $\NL_4< 26$};
\draw[-stealth,Navy,line width=1pt] (inveq.west)--(nl26.east) node[sloped,midway,above]{(2) Nonlinearity }node[sloped,midway,below]{$\geq 26$};

\node[draw,rectangle split, rectangle split parts=2, rectangle split part fill={Navy!30,Blue!20}, text width=4cm,align=center,rounded corners] (oddw) at (0,-4)
{ $\{f+\delta_a, g+\delta_a\}$
\nodepart{two}
$\delta_a$ Dirac function, $a \in \fdm$\\
 512 odd weight functions};
\draw[-stealth,Navy,line width=1pt] (nl26.west)--(oddw.east) node[sloped,midway,above]{Dirac translation};

\node[draw,rectangle split, rectangle split parts=2, rectangle split part fill={ForestGreen!50,ForestGreen!20}, text width=4cm,align=center,rounded corners] (rho) at (0,-8)
{ $\rho(4,8)=26$
\nodepart{two}
 No odd weight function with  $\NL_4\geq 27$};
\draw[-stealth,Navy,line width=1pt] (oddw.south)--(rho.north) node[sloped,midway,above]{\vbox{\hbox{(3)} \hbox{{Nonlinearity}}}}node[sloped,midway,below]{$\geq 27$};
 \node[draw,rectangle split, rectangle split parts=2, rectangle split part fill={ForestGreen!50,ForestGreen!20}, text width=4cm,align=center,rounded corners] (rho6) at (+14,-8)
{$\rho_6(4,8)=26$
 \nodepart{two}
 No function with $\NL_4\geq 28$
 in $\rmc(6,8)$
 };
 \draw[-stealth,Navy,line width=1pt] (inveq.south)--(rho6.north)
 node[sloped,midway,above]{\vbox{\hbox{(1)} \hbox{{Nonlinearity}}}}
 node[sloped,midway,below]{$\geq 28$};

%
%
%
%
\end{tikzpicture}
\caption{Strategy to compute $\rho(4,8)$}
\label{FIG}
\end{figure}
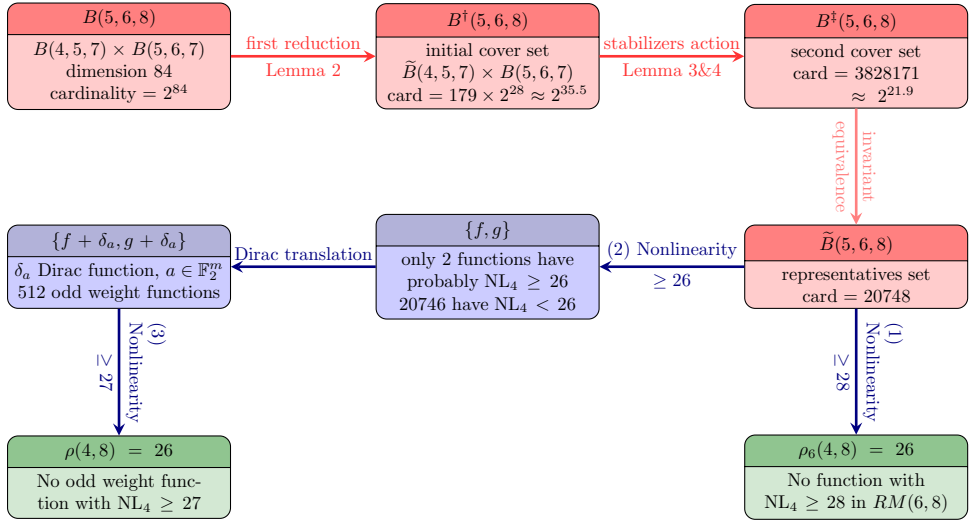

\section{Conclusion}
We have determine the covering radius of $\rmc(4,8)$ from the classification of $\boole(5,6,8)$. It is not obvious how to apply our method to obtain the covering radii of the second and third order Reed-Muller in $8$ variables. However, we believe that our approach can help to improve lower bounds in these open cases. 
\bibliographystyle{plain} 

\end{document}